\documentclass[12pt]{article} 

\usepackage{amsmath}
\usepackage{amsthm}
\usepackage{amsfonts}
\usepackage{mathrsfs}
\usepackage{stmaryrd}
\usepackage{setspace}
\usepackage{fullpage}
\usepackage{amssymb}
\usepackage{breqn}
\usepackage{enumitem}
\usepackage{bbold} 
\usepackage{authblk}
\usepackage{comment}
\usepackage{hyperref}
\usepackage{pgf,tikz}
\usepackage{graphicx}
\usepackage{colonequals}
\usetikzlibrary{decorations.pathreplacing,arrows}
\tikzstyle{vertex}=[circle,draw=black,fill=black,inner sep=0,minimum size=3pt,text=white,font=\footnotesize]

\newtheorem*{thm*}{Theorem}
\newtheorem{thm}{Theorem}

\newtheorem{lemma}[thm]{Lemma}
\newtheorem{conjecture}[thm]{Conjecture}

\newtheorem{prop}[thm]{Proposition}

\newtheorem{clm}[thm]{Claim}
\newtheorem*{proposition*}{Proposition}

\newcommand\cA{{\mathcal A}}
\newcommand\cB{{\mathcal B}}
\newcommand\cC{{\mathcal C}}

\newcommand\cF{{\mathcal F}}
\newcommand\cG{{\mathcal G}}

\newcommand\cR{{\mathcal R}}

\newcommand{\ignore}[1]{}

\title{A note on strongly and totally chain intersecting families}

\author{D\'aniel Gerbner\thanks{Research supported by the National Research, Development and Innovation Office -- NKFIH under the grants K 116769, KH 130371 and SNN 129364.}
\\
\small Alfr\'ed R\'enyi Institute of Mathematics, \\
\small \texttt{gerbner@renyi.hu}}
\date{}

\begin{document}

\maketitle

\begin{abstract} 
    Bern\'ath and Gerbner in 2007 introduced $(p,q)$-chain intersecting families of subsets of an $n$-element underlying set. Those have the property that for any $p$-chain $A_1\subsetneq A_2\subsetneq \dots \subsetneq A_p$ and $q$-chain $B_1\subsetneq B_2\subsetneq \dots \subsetneq B_q$, we have $A_p\cap B_q\neq \emptyset$. Bern\'ath and Gerbner determined the largest cardinality of such families. They also introduced strongly $(p,q)$-chain intersecting families, where $A_p\cap B_1\neq \emptyset$ and totally $(p,q)$-chain intersecting families, where $A_1\cap B_1\neq \emptyset$. They obtained some partial results on the maximum cardinality of such families. We extend those results by determining the largest cardinality of strongly $(p,q)$-chain intersecting families if $n$ is sufficiently large, and by determining the largest cardinality of totally $(2,2)$-chain intersecting families.
\end{abstract}

	
	

\section{Introduction}

Let $[n]=\{1,\dots,n\}$ be our underlying set. We are going to study families of subsets of $[n]$. As we only deal with such families, we will simply call them families and omit describing the underlying set.
One of the fundamental results in extremal finite set theory is the theorem of Sperner \cite{S}, which states that the largest cardinality of a family without a member containing another member of the family is $\binom{n}{\lfloor n/2\rfloor}$. Such family is called a \textit{Sperner family} or \textit{antichain}. An example of a Sperner family of cardinality  $\binom{n}{\lfloor n/2\rfloor}$ is the family of all the $\lfloor n/2\rfloor$-element sets. A family of all the sets of a given size is called a \textit{level}.

A \textit{chain of length $k$} is a family of $k$ sets pairwise in containment relation, i.e., $F_1\subsetneq\dots\subsetneq F_k$, thus an antichain shares at most one member with any chain. A natural generalization is to assume that there is no chain of length more than $k$ in a family $\cF$. Such families are called \textit{$k$-Sperner} families, their largest cardinality was determined by Erd\H os \cite{Er}. It is the cardinality of the \textit{middle $k-1$ levels}, which are the levels of size at least $\lfloor (n-k+1)/2\rfloor$ and at most $\lfloor (n+k-1)/2\rfloor$.

Another well-studied class of properties arises from the intersection of sets. A family is called \textit{intersecting} if any two members share at least one element. Erd\H os, Ko and Rado \cite{EKR} determined the largest cardinality of intersecting families and $r$-uniform intersecting families. The first is trivial: at most one of a set and its complement can be in an intersecting family, giving the upper bound $2^{n-1}$, and there are several intersecting families of cardinality $2^{n-1}$. We mention two of them: all the sets containing a fixed element, and all the sets of size more than $n/2$, together with the sets of size $n/2$ containing a fixed element, if $n$ is even. In the $r$-uniform case, the largest cardinality turns out to be $\binom{n-1}{r-1}$ if $r\le n/2$, with equality in the case of all the $r$-element sets containing a fixed element (if $r>n/2$, then every pair of $r$-element sets intersect).

There are several variants of the above problems, see \cite{gp} for a collection of them. It is a natural idea to combine Sperner-type and intersection properties. Intersecting Sperner and intersecting $k$-Sperner families were studied e.g. in \cite{milne,fra,gerbner,gmt}.

Bern\'ath and Gerbner \cite{BG} introduced a different combination of these properties. The family can contain long chains and disjoint sets, but cannot contain disjoint sets on the top of long chains. More precisely, a family $\cF$ is \textit{$(p,q)$-chain intersecting} if $\cF$ does not contain sets $A_1\subsetneq A_2\subsetneq \dots \subsetneq A_p$ and $q$-chain $B_1\subsetneq B_2\subsetneq \dots \subsetneq B_q$ with $A_p\cap B_q= \emptyset$.

For a positive integer $z$, we call the family consisting of all the sets of size at least $n-z+1$ the \textit{upper $z$ levels}. The \textit{upper $z+1/2$ levels} is the upper $z$ levels plus  the sets of size $n-z$
containing a fixed element.

\begin{thm}[Bern\'ath and Gerbner \cite{BG}]\label{bege} The largest cardinality of a $(p,q)$-chain-intersecting family is equal to the
cardinality of the upper $(n+p+q-1)/2$ levels.
\end{thm}

Bern\'ath and Gerbner \cite{BG} also introduced two variants of chain intersecting families; when the bottom sets in the chains have to intersect and when the top of one chain has to intersect the bottom of the other chain. More precisely, a family $\cF$ is totally \textit{$(p,q)$-chain intersecting} if $\cF$ does not contain sets $A_1\subsetneq A_2\subsetneq \dots \subsetneq A_p$ and $q$-chain $B_1\subsetneq B_2\subsetneq \dots \subsetneq B_q$ with $A_1\cap B_1= \emptyset$, and strongly \textit{$(p,q)$-chain intersecting} if $\cF$ does not contain sets $A_1\subsetneq A_2\subsetneq \dots \subsetneq A_p$ and $q$-chain $B_1\subsetneq B_2\subsetneq \dots \subsetneq B_q$ with $A_p\cap B_1= \emptyset$. 

\begin{conjecture}[Bern\'ath and Gerbner \cite{BG}]\label{conj1}
If $\cF$ is a strongly $(p, q)$-chain-intersecting, then $\cF\le \max\{|\cR_1|,\cR_2|\}$, where $\cR_1$ is the upper $(n+p)/2$ levels and $\cR_2$ is the middle $q-1$ largest levels.
\end{conjecture}

They proved this conjecture in some cases.

\begin{prop}[Bern\'ath and Gerbner \cite{BG}]\label{prop1}
Conjecture \ref{conj1} holds if $n+p$ is even or $p\ge q$ or $p=1$. 
\end{prop}

Our first result proves this conjecture for sufficiently large $n$. Note that in this case $|\cR_1|>|\cR_2|$.

\begin{thm}\label{thm1}
Let $p,q$ be positive integers and $n$ be an integer sufficiently large. Then Conjecture \ref{conj1} holds.
\end{thm}

Let us continue with totally $(p, q)$-chain-intersecting families $\cF$. Observe that if $p=1$, then Proposition \ref{prop1} implies that $|\cF|\le 2^{n-1}$, which is sharp as shown by any intersecting family of that cardinality. Here we deal with the first open case.

\begin{thm}\label{thm2}
The largest totally $(2,2)$-chain intersecting family has cardinality $2^{n-1}$.
\end{thm}

Finally, we state a conjecture regarding totally $(p, q)$-chain-intersecting families.  Let $\cF_q(i)=\{F\subset 2^{[n]}: 1\in F, |F|\le i\} \cup \{F\subset 2^{[n]}: i-q+1 \le |F|\le i\}$. In this family every chain of length $q$ contains a set of size at most $i-q$, thus all the members of such a chain contain $1$. This implies $\cF_q(i)$ is totally $(p,q)$-intersecting for every $i$, if $p\ge q$.

\begin{conjecture}\label{conj2} If $\cF$ is a totally  $(p,q)$-chain intersecting family and $p\ge q$, then $$|\cF|\le \max\{ |\cF_q(i)|,|\cR|\},$$ where $1\le i\le n$ and $\cR$ denotes the middle $p-1$ levels.
\end{conjecture}

\section{Preliminaries}\label{two}

Let us start with introducing the so-called \textit{permutation method}. Given and a permutation $\alpha$ on $[n]$ and a set $G\subset [n]$, we let $\alpha(G)=\{\alpha(x): \,x\in G\}$. Given a family $\cG$, we let $\alpha(\cG)=\{\alpha(G):G\in \cG\}$. Let $g_i$ denote the number of sets of size $i$ in $\cG$. Let us consider the following double sum.

\[\sum_{\alpha}\sum_{F\in \cF\cap \alpha(\cG)}\frac{1}{g_{|F|}}\binom{n}{|F|}=\sum_{F\in \cF}\sum_{\alpha:\,F\in \alpha(\cG)}\frac{1}{g_{|F|}}\binom{n}{|F|},\]

where $\alpha$ runs through every permutation of $[n]$. For every $F$ there are $|F|!(n-|F|)!$ permutations that map a given $|F|$-element set to $F$, thus there are $|F|!(n-|F|)!g_{|F|}$ permutations that map a member of $\cG$ to $F$. Therefore, the right hand side of the above equation is $|\cF|n!$. If we give an upper bound on $\sum_{F\in \cF\cap \alpha(\cG)}\frac{1}{g_{|F|}}$ for every $\alpha$, then we obtain an upper bound on $|\cF|$.

This method works for every family $\cG$, but few families give sharp upper bounds. For $k$-Sperner families a \textit{full chain}, i.e., a chain of maximal length $n+1$ can be used to obtain the theorem of Erd\H os \cite{Er}. The first application of the permutation method was the case $k=1$ in \cite{lubell}. We note that if $\frac{1}{g_{|F|}}\binom{n}{|F|}$ is replaced by a different weight function $w(|F|)$, then we obtain a bound $\sum_{F\in \cF}w'(|F|)$ for some $w'$. In particular, \cite{lubell} proved a stronger result than Sperner's theorem. Here we will always use this weight function, and we will simply say \textit{weight of a family $\cF$} instead of $\sum_{F\in\cF}\frac{1}{g_{|F|}}\binom{n}{|F|}$.

We will consider two other families as $\cG$, just like in \cite{BG}. The \textit{chain-pair} consists of a full chain $\cC$ and the chain consisting of the complements of the members of $\cC$. We denote the complement of a set $A$ by $\overline{A}$.

The other family is a bit more complicated. We consider $[n]$ cyclically, i.e., $1$ comes after $n$. An \textit{interval} consists of consecutive elements. An interval $[i,j]$ is of form $\{i,i+1,\dots,i+j\}$, where addition is modulo $n$. Note that $\emptyset$ and $[n]$ are also intervals. The permutation method with $\cG$ being the family of intervals is called the \textit{circle method} and was introduced in \cite{kat}. When we talk about a subfamily of the intervals, we simply say that the family is on the circle.

Let us describe how the permutation method was used in \cite{BG}. First we need to bound $\sum_{F\in \cF\cap \alpha(\cG)}\frac{1}{g_{|F|}}\binom{n}{|F|}$ on the chain-pair or the circle. In each case, the bound can be easily found on the chain-pair. In particular, for strongly $(p,q)$-chain intersecting families, i.e., in the setting of Theorem \ref{bege}, the upper $(n+p+q-1)/2$ levels have the maximum on the chain-pair. This gives a sharp bound if $n+p+q-1$
 is even. However, if $n+p+q-1$ is odd, then we have a set on level $(n-p-q)/2$. On the chain-pair, this gives the weight $\binom{n}{(n-p-q)/2}/2$, thus the upper bound is the cardinality of the top $(n+p+q-2)/2$ levels plus $\binom{n}{(n-p-q)/2}/2$. On the other hand, the conjectured bound is the cardinality of the top $(n+p+q-2)/2$ levels plus $\binom{n-1}{(n-p-q-2)/2}$. In this case the circle helps, where again the top $(n+p+q-1)/2$ levels have the maximum weight.

More generally, if the optimal family on the chain-pair contains both sets from a level $i$, they give the weight $\binom{n}{i}$, which corresponds to the cardinality of the full level $i$. If the  optimal family on the chain-pair contains one set from level $i<n/2$ (the one containing our fixed element), this corresponds to half that cardinality, which is larger than the cardinality of an intersecting family on that level. Analogously, if $i>n/2$, we may look for the complements of the members of an intersecting family. Therefore, the chain-pair may give a sharp result if we have full level everywhere, except maybe level $n/2$. On the other hand, on the circle the half level$i$ where $i<n/2$ consists of $i$ sets, which gives a sharp bound in these cases.

One can see that the weighted analogue of Conjecture \ref{conj1} holds on the chain-pair (we will show it in the proof of Lemma \ref{lemm1}). This gives most of the results mentioned in Proposition \ref{prop1}. In fact this gives a bit more: if $n+p$ is odd and the middle $q$ levels are optimal on the chain-pair for the weighted problem, i.e., the middle $q-1$ levels are larger than the upper $(n+p-1)/2$ levels plus $\binom{n}{n-p+2}/2$, then again we obtain a sharp bound. 

\smallskip

We will use a theorem of Hilton \cite{hilton}. We say that two families $\cF$ and $\cF'$ are \textit{cross-Sperner} if there are no members $F\in \cF$ and $F'\in\cF'$ with $F\subset F'$ or $F'\subset F$.

\begin{thm}[Hilton \cite{hilton}]\label{hilt}
    If $\cF$ and $\cF'$ are complement-free cross-Sperner families, then $|\cF|+|\cF'|\le 2^{n-1}$.
\end{thm}

\section{Proofs}

We say that a family $\cF$ is \textit{$r$-complementing-chain-pair-free} if it does not contain sets $F_1\subsetneq F_2\subsetneq\dots\subsetneq F_r$ together with their complements. 
Clearly a $(p,q)$-chain intersecting family is $(p+q-1)$-complementing-chain-pair-free. Theorem \ref{bege} uses this observation, more precisely the following lemma.

\begin{lemma}[Bern\'ath and Gerbner \cite{BG}]\label{begelem}
The largest weight of an $r$-complementing-chain-pair-free family on the circle is equal to the weight of the upper $(n+r)/2$ levels on the circle.
\end{lemma}

\begin{lemma}\label{lemm1}
If $n$ is sufficiently large, then the largest weight of a strongly $(p,q)$-chain intersecting family on the circle is equal to the weight of the upper $(n+p)/2$ levels on the circle.
\end{lemma}

\begin{proof}
As we have mentioned, the statement holds on the chain-pair. More precisely, if $\cC$ is a chain-pair and $\cF_0$ is a strongly $(p,q)$-chain intersecting subfamily of $\cC$, then we have two options. If a chain $\cA$ in $\cC$ contains at least $q$ members, then the other chain $\cB$ can have at most $p-1$ members that do not intersect the bottom of $\cA$. Therefore, $\cF_0$ is the union of an intersecting (thus complement-free) family and a $(p-1)$-Sperner family on one chain. The maximal weight is obtained when the $p-1$ sets are in the middle of $\cB$ and we pick one set from each complement-pair. The upper $(n+p)/2$ levels give equality here, let $x$ denote their total weight. In the other case both chains contain at most $q-1$ sets, the maximal weight is obtained when we take $q-1$ levels in the middle, let $y$ denote the total weight of them. Clearly $x>2^{n-1}$ and it is well known that $y=\Theta(2^n/\sqrt{n})$, thus $x>y$ if $n$ is sufficiently large.

Let $\cF$ be a strongly $(p,q)$-chain intersecting family on the circle. If $\cF$ is $p$-complementing-chain-pair-free, then we are done by Lemma \ref{begelem}. Assume that $\cF$ contains $A_1\subsetneq A_2\subsetneq \dots A_p$ together with their complements. Let $\cA$ denote a chain on the circle containing $A_1,\dots,A_p$, $\cB$ denote the complement chain and $\cC$ denote the chain-pair $\cA\cup\cB$. If $|\cA\cap \cF|\ge q$, then the bottom member of $\cF$ in $\cA$ is a subset of $A_1$, thus disjoint from $\overline{A_1}$, which is the top of a chain of length $p$, contradicting the strongly $(p,q)$-chain intersecting. Analogously, $|\cB\cap \cF|\ge q$ is also impossible, thus we have $|\cC\cap \cF|\le 2q-2$.

Let $\cC_i$ denote the chain-pair we obtain from $\cC$ by replacing each interval $[a,b]$ by $[a+i,b+i]$. Observe that each interval on the circle except for $\emptyset$ and $[n]$ is contained in exactly two of the chain-pairs $\cC_i$, $0\le i\le n-1$. Clearly, $\emptyset$ and $[n]$ are in every chain-pair and have weight 1. Therefore, $2w(\cF\cap\cG)=\sum_{i=0}^{n-1}w(\cC_i)\le (n-1)x+y+2n-4$. 

Observe that $nx/2$ is the upper bound on the weight we get by using the chain-pairs without any further ideas. This is the weight of the upper $(n+p-1)/2$ levels plus half the weight of the next level (level $(n-p+1)/2$) on the circle, i.e., $n\binom{n}{(n-p+1)/2}/2$. The conjectured optimal construction is the upper $(n+p)/2$ levels, thus its weight is the weight of the upper $(n+p-1)/2$ levels plus $(n-p+1)\binom{n}{(n-p+1)/2}/2$. Therefore, this weight as exactly $nx/2-(p-1)\binom{n}{(n-p+1)/2}/2$.
The upper bound $(n-1)x+y+2n-4$ we obtained is at most $nx/2-c2^n$ for some constant $c$. If $n$ is sufficiently large, then $(p-1)\binom{n}{(n-p+1)/2}/2<c2^n$, completing the proof.
\end{proof}

Theorem \ref{thm1} follows from the above lemma by applying the permutation method as described in Section \ref{two}.
Let us continue with Theorem \ref{thm2}. Recall that it states that $(2,2)$-totally chain intersecting families have cardinality at most $2^{n-1}$.

\begin{proof}[Proof of Theorem \ref{thm2}]
    Let $\cF$ be a $(2,2)$-totally chain intersecting family. Let $\cG$ be the subfamily of sets that are both maximal and minimal in $\cF$. For each pair $G,\overline{G}$ of complement sets in $\cG$, we pick one of them
    and place these sets to $\cG_1$. Let $\cG_2=\cG\setminus \cG_1$ and let $\cF_1=\cF\setminus \cG_2$.

    We claim that $\cF_1$ is complement-free. Indeed, assume that $F,\overline{F}\in \cF_1$. If one of them, say $F$ is in $\cG$, then $F\in\cG_1$, but then $\overline{F}\in\cG_2$, thus $\overline{F}\not\in \cF_1$, a contradiction.
    Therefore, $F,\overline{F}\not\in \cG$, thus both $F$ and $\overline{F}$ are contained in some chain of length 2. The bottom sets in those chains must be disjoint, creating a forbidden configuration.

    Clearly, $\cG_2$ is complement-free. $\cF_1$ and $\cG_2$ are cross-Sperner since the elements of $\cG_2$ are not related to any set in $\cF$. Therefore, we can apply Theorem \ref{hilt} to complete the proof.
\end{proof}

\bigskip

\textbf{Competing Interests}: The author has no relevant financial or non-financial interests to disclose.

\smallskip

\textbf{Data availability}: Data sharing not applicable to this article as no datasets were generated or analysed during the current study.

\end{document}